\documentclass[twocolumn]{jsiamletters}
\usepackage{graphicx}
\usepackage{amsmath,amssymb,amsthm}
\usepackage{bm}
\usepackage{color}

\newtheorem{Def}{Definition}

\newtheorem{Lem}[Def]{Lemma}

\usepackage{algorithmicx}
\usepackage[ruled]{algorithm}
\usepackage{algpseudocode}
\usepackage{color}

\group{
Algorithms for Matrix / Eigenvalue Problems and their Applications
}
\affiliation{The University of Electro-Communications}{1-5-1 Chofugaoka, Chofu, Tokyo, 182-8585, Japan}
\authorinfo{Tomoya Miyashita}{1}{m2231146@edu.cc.uec.ac.jp}
\authorinfo{Shuhei Kudo}{1}{shuhei-kudo@uec.ac.jp}
\authorinfo{Yusaku Yamamoto}{1*}{yusaku.yamamoto@uec.ac.jp}
\email{yusaku.yamamoto@uec.ac.jp}
\title{Roundoff error analysis of the double exponential formula-based method for the matrix sign function}
\abstract{
In this paper, we perform a roundoff error analysis of an integration-based method for computing the matrix sign function recently proposed by Nakaya and Tanaka. The method expresses the matrix sign function using an integral representation and computes the integral numerically by the double-exponential formula. While the method has large-grain parallelism and works well for well-conditioned matrices, its accuracy deteriorates when the input matrix is ill-conditioned or highly nonnormal. We investigate the reason for this phenomenon by a detailed roundoff error analysis.
}
\begin{document}

\maketitle

\section{Introduction}
Let $A\in\mathbb{R}^{n\times n}$ be a real matrix and $A=XJX^{-1}$ be its Jordan decomposition. Here, $X\in\mathbb{R}^{n\times n}$ is a nonsingular matrix and $J=J_{n_1}(\lambda_1)\oplus\cdots\oplus J_{n_r}(\lambda_r)$, where $J_{n_i}(\lambda_i)$ ($i=1, \ldots r$) is a Jordan cell of size $n_i\times n_i$ with eigenvalue $\lambda_i$ and $n_1+\cdots+n_r=n$. We assume that $A$ has no purely imaginary eigenvalues. Also, let $I_n$ be the identity matrix of order $n$ and ${\rm sign}(z)$, where $z\in\mathbb{C}$, be the (scalar) sign function defined by
\begin{equation}
{\rm sign}(z)=\begin{cases}
1, & {\rm Re}(z)>0, \\
-1, & {\rm Re}(z)<0.
\end{cases}
\end{equation}
Then, the {\it matrix sign function} \cite{Higham08} is defined by
\begin{equation}
{\rm sign}(A)\equiv X({\rm sign}(\lambda_1)I_{n_1}\oplus\cdots\oplus{\rm sign}(\lambda_r) I_{n_r})X^{-1}.
\label{eq:definition}
\end{equation}
Note that the matrix sign function is undefined when $A$ has purely imaginary eigenvalues. The matrix sign function has applications to the solution of the Sylvester equation and computation of eigendecomposition \cite{Nakatsukasa16}.

Several approaches have been proposed for computing the matrix sign function, including Schur's method, Newton's method and rational function approximation \cite{Higham08}. Recently, Nakaya and Tanaka proposed a new method for computing the matrix sign function using numerical quadrature \cite{Nakaya21}. The idea is to expresses ${\rm sign}(A)$ using an integral representation:
\begin{equation}
{\rm sign}(A)=\frac{2}{\pi}\int_0^{\infty}(t^2 I+A^2)^{-1}A\,dt
\label{eq:integral_representation}
\end{equation}
and compute it numerically with the double-exponential (DE) formula \cite{Takahashi74}. In the following, We refer to this method as the {\it DE-based method} for the matrix sign function. Since the computation of $(t^2 I+A^2)^{-1}$ can be done for each sample point independently, the method has large-grain parallelism and is well-suited for modern high performance computers. Detailed analysis of the discretization and truncation errors of the DE-based method, assuming exact arithmetic, is given in \cite{Nakaya21}. The reader is also referred to \cite{Miyashita22} for an analysis in the case where $A$ is diagonalizable.

According to our numerical experiments, the method works well when $A$ is a well conditioned and close-to-normal matrix and delivers results with accuracy comparable to that of Schur's method. However, we observed that as $A$ becomes ill-conditioned or deviates from normality (which means that $X$ becomes ill-conditioned), the numerical error of the method increases rapidly. This problem is not fixed even if we make the step size sufficiently small or make the truncation points sufficiently distant from the origin. Hence, we suppose that this degradation of accuracy originates not from discretization or truncation errors but from rounding errors.

In this paper, we present a roundoff error analysis of the DE-based method for the matrix sign function to find the cause of this accuracy degradation and look for a direction for possible improvement. For simplicity, we focus on the case where $A$ is diagonalizable and consider two main sources of rounding errors, namely, the error arising in the computation of $(t^2 I+A^2)^{-1}A$ at each sample point and the error that occurs when summing up the contributions from the sample points.

The rest of this paper is structured as follows. In Section 2, we review the DE-based method for the matrix sign function. Roundoff error analysis of this method is given in Section 3. Numerical results that illustrate the validity of the derived error bound are presented in Section 4. Finally, Section 5 provides some conclusion.

\section{DE-based method for the matrix sign function}
In the DE-based method, we approximate the integral \eqref{eq:integral_representation} using the DE formula. By letting $Y(t)\equiv (t^2I+A^2)^{-1}A$ and $\phi(x)=\exp(\frac{\pi}{2}\sinh x)$, we have
\begin{align}
{\rm sign}(A) &= \frac{2}{\pi}\int_0^{\infty}Y(t)dt = \frac{2}{\pi}\int_{-\infty}^{\infty}Y(\phi(x))\phi'(x)dx \nonumber \\
&\simeq \frac{2}{\pi}h\sum_{k=N^-}^{N^+}Y(\phi(kh))\phi'(kh),
\end{align}
where $h$ is the step size and $-N^->0$ and $N^+>0$ are the number of sample points in the negative and positive part of the $x$-axis, respectively. It is shown that by choosing $-N^-=N^+=N$ and $h=\log(8dN)/N$, where $d$ is some constant depending on $A$, the discretization and truncation errors of the DE-based method decrease almost exponentially with $N$ \cite{Nakaya21}. In this method, the most computationally intensive part is the calculation of $Y(\phi(kh))$ for $k=N^-, \ldots, N^+$. Since this can be done for each $k$ independently, the method has large-grain parallelism.

\section{Roundoff error analysis}
We present a roundoff error analysis of the DE-based method. We denote a quantity computed in floating-point arithmetic by $fl(\cdot)$ or by a symbol with a hat. ${\bf u}$ denotes the unit roundoff and $\gamma_m\equiv m{\bf u}/(1-m{\bf u})$. For a matrix $A=(a_{ij})$, $|A|$ means a matrix whose elements are $|a_{ij}|$. For matrices $A$ and $B$ of the same dimension, $A\le B$ means componentwise inequality. $\|\cdot\|_2$ and $\|\cdot\|_F$ denote the 2-norm and the Frobenius norm, respectively. The condition number of $A$ is denoted by $\kappa_2(A)$.

\subsection{Sources of roundoff errors}
The true roundoff error arising in the computation of the DE-based method can be written as follows.
\begin{align}
& E_{\rm true} \nonumber \\
&= \overline{fl}\!\left[\frac{2}{\pi}h\!\cdot\!fl\left\{\sum_{k=N^-}^{N^+}\!\overline{fl}\left(fl(Y(\overline{fl}(\phi(kh))))\!\cdot\!\overline{fl}(\phi'(kh))\right)\!\right\}\right] \nonumber \\
& \quad - \frac{2}{\pi}h\sum_k Y(\phi(kh))\phi'(kh).
\end{align}
To simplify the analysis, we assume that scalar functions such as $\phi(x)$ and $\phi'(x)=\frac{\pi}{2}\exp(\frac{\pi}{2}\sinh x)\cosh x$ can be computed without errors. We also assume that multiplications by scalars such as $\phi'(x)$ and $\frac{2}{\pi}h$ can be done without errors. This amounts to assuming that the computations denoted by $\overline{fl}(\cdot)$ can be done exactly. Under this assumption, we can write the total roundoff error as
\begin{align}
E &= \frac{2}{\pi}h\,fl\left(\sum_k fl(Y(\phi(kh)))\phi'(kh)\right) \nonumber \\
& \quad - \frac{2}{\pi}h\sum_k Y(\phi(kh))\phi'(kh) \nonumber \\
&= \frac{2}{\pi}h\cdot fl\left(\sum_k fl(Y(\phi(kh)))\phi'(kh)\right) \nonumber \\
& \quad - \frac{2}{\pi}h\sum_{k}fl(Y(\phi(kh)))\phi'(kh) \nonumber \\
& \quad + \frac{2}{\pi}h\sum_{k}fl(Y(\phi(kh)))\phi'(kh) \nonumber \\
& \quad - \frac{2}{\pi}h\sum_k Y(\phi(kh))\phi'(kh) \nonumber \\
&= \left\{\frac{2}{\pi}h\cdot fl\left(\sum_k fl(Y(\phi(kh)))\phi'(kh)\right) \right. \nonumber \\
& \quad\quad \left. - \frac{2}{\pi}h\sum_{k}fl(Y(\phi(kh)))\phi'(kh)\right\} \nonumber \\
& \quad + \frac{2}{\pi}h\sum_{k}\left(fl(Y(\phi(kh)))-Y(\phi(kh))\right)\phi'(kh) .
\end{align}
The last expression suggests that the total roundoff error consists of the following two parts.
\begin{itemize}
\item Errors in the computation of $Y(t)$: Let $\hat{Y}(t)\equiv fl(Y(t))$ and $\tilde{E}_1(t)\equiv \hat{Y}(t)-Y(t)$, where $t=\phi(kh)$ and $N^-\le k\le N^+$. The weighted sum of these errors,
\begin{equation}
E_1\equiv \frac{2}{\pi}h\sum_{k=N^-}^{N^+}|\tilde{E}_1(\phi(kh))|\phi'(kh),
\label{eq:E1}
\end{equation}
contributes to the total roundoff error.
\item Errors in the summation. Strictly speaking, the summand is $fl(Y(\phi(kh))\phi'(kh))$, but we substitute it with its exact counterpart, $Y(\phi(kh))\phi'(kh)$, for simplicity. This can be justified because their difference is $O({\bf u})$, as will be shown later (see \eqref{eq:hatXerror}), and therefore causes only $O({\bf u}^2)$ difference in the value of $E_2$. Thus, the error is defined as
\begin{align}
E_2 &\equiv \frac{2}{\pi}h\cdot fl\left(\sum_{k=N^-}^{N^+} Y(\phi(kh))\phi'(kh)\right) \nonumber \\
& \quad\quad - \frac{2}{\pi}h\sum_{k=N^-}^{N^+} Y(\phi(kh))\phi'(kh).
\label{eq:E2}
\end{align}
\end{itemize}
In the following, we evaluate $E_1$ and $E_2$ separately.

\subsection{Evaluation of $\tilde{E}_1(t)=\hat{Y}(t)-Y(t)$}
Let $B\equiv t^2 I+A^2$ and denote the $j$th column of $A$, $Y(t)$ and $\hat{Y}(t)$ by ${\bf a}_j$, ${\bf y}_j$ and $\hat{\bf y}_j$, respectively. Then, ${\bf y}_j$ is computed as the solution of the linear simultaneous equations $B{\bf y}_j={\bf a}_j$. Thus, $\hat{\bf y}_j-{\bf y}_j$ consists of two parts, the error in the formation of $B$ and that in the solution of the linear simultaneous equations. Denote the former error by $\Delta B'$. Then, if we consider only the error in the computation of $A^2$ and ignore the error arising from the addition of $t^2 I$, we have $|\Delta B'|\le\gamma_n|A|^2$ from the result of the standard error analysis \cite[\S 3.5]{Higham02} and therefore
\begin{equation}
\|\Delta B^{\prime}\|_F \le \gamma_n\|A\|_F^2.
\label{eq:deltaBbound}
\end{equation}
Now, $\hat{\bf y}_j$ is obtained by solving the linear simultaneous equation with the coefficient matrix $\tilde{B}\equiv B+\Delta B'$ using Gaussian elimination with partial pivoting in floating-point arithmetic. In that case, it is well known that $\hat{\bf y}_j$ satisfies the following equation \cite[Theorem 9.4]{Higham02}:
\begin{equation}
(\tilde{B}+\Delta B_j^{\prime\prime})\hat{\bf y}_j={\bf a}_j, \quad |\Delta B_j^{\prime\prime}|\le \gamma_{3n}|\hat{L}|\,|\hat{U}|,
\label{eq:Higham9-4}
\end{equation}
where $\hat{L}$ and $\hat{U}$ are computed LU factors of $\tilde{B}$ and $\Delta B_j^{\prime\prime}$ is the backward error in the solution of the linear simultaneous equation. Now, we evaluate $\|\Delta B_j^{\prime\prime}\|_F$. First, since $|\hat{l}_{ij}|\le 1$, we have $\|\hat{L}\|_F\le n$. Next, let the coefficient matrix in the $k$th step of the Gaussian elimination be $\tilde{B}^{(k)}=(\tilde{b}_{ij}^{(k)})$ and define the {\it growth factor} as
\begin{equation}
\hat{\rho}_n = \frac{\max_{i,j,k}|\tilde{b}_{i,j}^{(k)}|}{\max_{i,j}|\tilde{b}_{i,j}|}.
\end{equation}
Then,
\begin{equation}
|\hat{u}_{i,j}|=|\tilde{b}_{i,j}^{(i)}| \le \hat{\rho}_n\max_{i',j'}|\tilde{b}_{i',j'}| \le \hat{\rho}_n\|\tilde{B}\|_2
\end{equation}
and we have $\|\hat{U}\|_F\le n\hat{\rho}_n\|\tilde{B}\|_2$. From these results, it follows that
\begin{align}
\|\Delta B_j^{\prime\prime}\|_F &\le n^2\gamma_{3n}\hat{\rho}_n\|\tilde{B}\|_2 \nonumber \\
&\le n^2\gamma_{3n}\hat{\rho}_n(\|B\|_2+\gamma_n\|A\|_F^2) \nonumber \\
&= n^2\gamma_{3n}\hat{\rho}_n\|t^2I + A^2\|_2 + O({\bf u}^2). \nonumber \\
&\simeq n^2\gamma_{3n}\hat{\rho}_n\|t^2I + A^2\|_2.
\label{eq:DeltaAbound}
\end{align}
It is known that the growth factor $\hat{\rho}_n$ is almost independent of $n$ and is typically around 10 \cite[\S 9.4]{Higham02}. Combining \eqref{eq:deltaBbound}, \eqref{eq:Higham9-4} and \eqref{eq:DeltaAbound}, we know that $\hat{\bf y}_j$ satisfies
\begin{align}
& (B+\Delta B_j)\hat{\bf y}_j={\bf a}_j, \\
& \|\Delta B_j\|_F \le \gamma_n\|A\|_F^2+n^2\gamma_{3n}\hat{\rho}_n\|t^2I + A^2\|_2,
\end{align}
where $\Delta B_j = \Delta B'+ \Delta B_j^{\prime\prime}$. Now, assume that$\|B^{-1}\|_2\|\Delta B_j\|_2\le 1/2$. Applying the perturbation theory for linear simultaneous equations \cite[Theorem 7.2]{Higham02} gives
\begin{align}
\|\hat{\bf y}_j-{\bf y}_j\| &\le \frac{\|B^{-1}\|_2\|\Delta B_j\|_2}{1-\|B^{-1}\|_2\|\Delta B_j\|_2}\,\|{\bf y}_j\| \nonumber \\
&\le 2\|B^{-1}\|_2\|\Delta B_j\|_2\|{\bf y}_j\| \nonumber \\
&\le 2\|(t^2I+A^2)^{-1}\|_2 \nonumber \\
& \quad \times(\gamma_n\|A\|_F^2+n^2\gamma_{3n}\hat{\rho}_n\|t^2I + A^2\|_2)\|{\bf y}_j\| . \nonumber
\end{align}
From this, it is immediate to show that
\begin{align}
& \|\tilde{E}_1(t)\|_F=\|\hat{Y}(t)-Y(t)\|_F \nonumber \\
&\le 2\|(t^2I+A^2)^{-1}\|_2(\gamma_n\|A\|_F^2+n^2\gamma_{3n}\hat{\rho}_n\|t^2I + A^2\|_2) \nonumber \\
& \quad \times\|(t^2I+A^2)^{-1}A\|_F.
\label{eq:hatXerror}
\end{align}

\subsection{Evaluation of $E_1$}
Now we evaluate $E_1$ using \eqref{eq:E1} and \eqref{eq:hatXerror}. From the assumption of diagonalizability, $A$ can be written as $A=X\Lambda X^{-1}$, where $\Lambda={\rm diag}(\lambda_1,\ldots,\lambda_n)$. Hence,
\begin{align}
\|A\|_F^2 &= \|X\Lambda X^{-1}\|_F^2 \nonumber \\
&\le (\|X\|_2\|\Lambda\|_F\|X^{-1}\|_2)^2 = (\kappa_2(X))^2\|\Lambda\|_F^2,
\label{eq:AF2}
\end{align}
where we used $\|AB\|_F \le \|A\|_2\|B\|_F$. Similarly, we have
\begin{align}
\|t^2 I + A^2\|_2 &\le \kappa_2(X)\|t^2I+\Lambda^2\|_2,
\label{eq:YF1} \\
\|(t^2 I + A^2)^{-1}\|_2 &\le \kappa_2(X)\|(t^2I+\Lambda^2)^{-1}\|_2, \label{eq:YF2} \\
\|(t^2I+A^2)^{-1}A\|_F &\le \kappa_2(X)\|(t^2I+\Lambda^2)^{-1}\Lambda\|_F. \label{eq:YF3}
\end{align}
Substituting \eqref{eq:AF2} through \eqref{eq:YF3} into \eqref{eq:hatXerror} gives
\begin{align}
& \|\tilde{E}_1(t)\|_F \nonumber \\
&\le 2\left\{\gamma_n(\kappa_2(X))^4\|\Lambda\|_F^2+n^2\gamma_{3n}\hat{\rho}_n(\kappa_2(X))^3\|t^2I+\Lambda^2\|_2\right\} \nonumber \\
& \quad \times\|(t^2I+\Lambda^2)^{-1}\|_2 \|(t^2I+\Lambda^2)^{-1}\Lambda\|_F.
\label{eq:Yterror}
\end{align}
From \eqref{eq:E1} and \eqref{eq:Yterror}, the contribution to the total roundoff error can be computed as
\begin{align}
& \|E_1\|_F \le \frac{2}{\pi}h\sum_{k=N^-}^{N^+}\|\tilde{E}_1(\phi(kh))\|_F\phi'(kh) \nonumber \\
&\simeq \frac{2}{\pi}\int_{-\infty}^{\infty}\|\tilde{E}_1(\phi(x))\|_F\phi'(x)\,dx 
= \frac{2}{\pi}\int_0^{\infty}\|\tilde{E}_1(t)\|_F\,dt \nonumber \\
&\le \frac{4}{\pi}\int_0^{\infty}\left\{\gamma_n(\kappa_2(X))^4\|\Lambda\|_F^2 \right.\nonumber \\
&\quad\quad\quad\quad\quad \left.+n^2\gamma_{3n}\hat{\rho}_n(\kappa_2(X))^3\|t^2I+\Lambda^2\|_2\right\} \nonumber \\
&\quad\quad\quad\quad \times \|(t^2I+\Lambda^2)^{-1}\|_2 \|(t^2I+\Lambda^2)^{-1}\Lambda\|_F\,dt.
\label{eq:Yphibound}
\end{align}
Let us write the integrand of the last integral as $e_{n,X,\Lambda}(t)$. To evaluate the integral, we let $t_1=\sqrt{2}\|\Lambda\|_F$ and divide the integration interval into two parts, $[0,t_1]$ and $[t_1,\infty)$. When $t\ge t_1$, we have
\begin{align}
\|\Lambda\|_F^2 &\le \frac{t^2}{2}, \label{eq:LFbound} \\
\|t^2I+\Lambda^2\|_2 &= \max_{1\le j\le n}|t^2+\lambda_j^2| \le \max_{1\le j\le n}(t^2+|\lambda_j|^2) \nonumber \\
&\le t^2 + \frac{t^2}{2} = \frac{3}{2}t^2, \\
\|(t^2I+\Lambda^2)^{-1}\|_2 &= \max_{1\le j\le n}\frac{1}{|t^2+\lambda_j^2|} \le \max_{1\le j\le n}\frac{1}{t^2-|\lambda_j|^2} \nonumber \\
&\le \frac{1}{t^2-\frac{1}{2}t^2} =\frac{2}{t^2}, \\
\|(t^2I+\Lambda^2)^{-1}\Lambda\|_F &= \sqrt{\sum_{j=1}^n\frac{|\lambda_j|^2}{|t^2+\lambda_j^2|^2}} \le \frac{2}{t^2}\sqrt{\sum_{j=1}^n|\lambda_j|^2} \nonumber \\
&= \frac{2}{t^2}\|\Lambda\|_F.
\label{eq:integrandbound3}
\end{align}
Now, let
\begin{equation}
c_{n,X,\Lambda}=\gamma_n(\kappa_2(X))^4+3n^2\gamma_{3n}\hat{\rho}_n(\kappa_2(X))^3.
\end{equation}
Then, we have from \eqref{eq:LFbound} through \eqref{eq:integrandbound3},
\begin{align}
0 \le \int_{t_1}^{\infty}e_{n,X,\Lambda}(t)\,dt
&\le c_{n,X,\Lambda}\int_{t_1}^{\infty}\frac{t^2}{2}\cdot\frac{2}{t^2}\cdot\frac{2}{t^2}\|\Lambda\|_F\,dt \nonumber \\
&= c_{n,X,\Lambda}\|\Lambda\|_F\int_{t_1}^{\infty}\frac{2}{t^2}\,dt \nonumber \\
&= c_{n,X,\Lambda}\|\Lambda\|_F\cdot\frac{\sqrt{2}}{\|\Lambda\|_F} \nonumber \\
&= \sqrt{2}\,c_{n,X,\Lambda}.
\label{eq:integralbound3b}
\end{align}
On the other hand, when $0\le t\le t_1$,
\begin{align}
\|t^2I+\Lambda^2\|_2 &\le \max_{1\le j\le n}(t^2+|\lambda_j|^2) \le 3\|\Lambda\|_F^2, \\
\|(t^2I+\Lambda^2)^{-1}\|_2 &= \max_{1\le j\le n}\frac{1}{|t^2+\lambda_j^2|}, \\
\|(t^2I+\Lambda^2)^{-1}\Lambda\|_F &\le \|\Lambda\|_F\|(t^2I+\Lambda^2)^{-1}\|_2 \nonumber \\
&\le \|\Lambda\|_F\max_{1\le j\le n}\frac{1}{|t^2+\lambda_j^2|}. \label{eq:integrandbound6}
\end{align}
Hence,
\begin{align}
0 &\le \int_0^{t_1}e_{n,X,\Lambda}(t)\,dt \le c_{n,X,\Lambda}\|\Lambda\|_F^3\int_0^{t_1}\max_{1\le j\le n}\frac{dt}{|t^2+\lambda_j^2|^2} \nonumber \\
&\quad\le c_{n,X,\Lambda}\|\Lambda\|_F^3\sum_{j=1}^n\int_0^{t_1}\frac{dt}{|t^2+\lambda_j^2|^2} \nonumber \\
&\quad\le \frac{1}{2}c_{n,X,\Lambda}\|\Lambda\|_F^3\sum_{j=1}^n\int_{-\infty}^{\infty}\frac{dt}{|t^2+\lambda_j^2|^2}.
\end{align}
Noting that ${\rm Re}(\lambda_j)\ne 0$, we can evaluate this integral as
\begin{equation}
\int_{-\infty}^{\infty}\frac{dt}{|t^2+\lambda_j^2|^2} = \frac{\pi}{2|\lambda_j|^2|{\rm Re}(\lambda_j)|}.
\label{eq:quadraticintegral}
\end{equation}
See Lemma \ref{Lemma1} in the Appendix. Thus, it follows that
\begin{align}
0&\le \int_0^{t_1}e_{n,X,\Lambda}(t)\,dt
\le \frac{\pi}{4}c_{n,X,\Lambda}\|\Lambda\|_F^3\sum_{j=1}^n \frac{1}{|\lambda_j|^2|{\rm Re}(\lambda_j)|} \nonumber \\
&= \frac{\pi}{4}c_{n,X,\Lambda}\|\Lambda\|_F^3\|\Lambda^{-1}|{\rm Re}(\Lambda)|^{-\frac{1}{2}}\|_F^2.
\label{eq:quadraticintegral2}
\end{align}
Combining \eqref{eq:Yphibound}, \eqref{eq:integralbound3b} and \eqref{eq:quadraticintegral2} gives
\begin{align}
\|E_1\|_F &\le c_{n,X,\Lambda}\left(\frac{4\sqrt{2}}{\pi}+\|\Lambda\|_F^3\|\Lambda^{-1}|{\rm Re}(\Lambda)|^{-\frac{1}{2}}\|_F^2\right) \nonumber \\
&=\left(\gamma_n(\kappa_2(X))^4+3n^2\gamma_{3n}\hat{\rho}_n(\kappa_2(X))^3\right) \nonumber \\&\quad \times\left(\frac{4\sqrt{2}}{\pi}+\|\Lambda\|_F^3\|\Lambda^{-1}|{\rm Re}(\Lambda)|^{-\frac{1}{2}}\|_F^2\right).
\label{TermwiseError}
\end{align}

\subsection{Evaluation of $E_2$}
Next, we evaluate $E_2$, the roundoff error in the summation. Let us consider the sum of $m$ matrices, $S_m=\sum_{i=1}^m T_i$, and denote the computed result by $\hat{S}_m$. Then, from the formula of roundoff error bound for scalar summation \cite[Problem 4.3]{Higham02}, $|\hat{S}_m-S_m|$ can be bounded as follows \cite{Yamamoto22}.
\begin{equation}
|\hat{S}_m-S_m| \le \gamma_{m-1}(|T_1|+|T_2|)+\sum_{i=3}^m\gamma_{m-i+1}|T_i|.
\end{equation}
Taking the Frobenius norm of both sides and replacing $\gamma_{m-i+1}$ with $\gamma_{m-1}$ for simplicity gives
\begin{equation}
\|\hat{S}_m-S_m\|_F \le \gamma_{m-1}\sum_{i=1}^m\|T_i\|_F. 
\end{equation}
By applying this result to \eqref{eq:E2} and writing $M=N_+-N_-$, we have
\begin{align}
\|E_2\|_F &=\frac{2}{\pi}h\gamma_{M}\sum_{k=N^-}^{N^+}\|Y(\phi(kh))\|_F\phi'(kh) \nonumber \\
&\simeq \frac{2}{\pi}\gamma_{M}\int_{-\infty}^{\infty}\|Y(\phi(x))\|_F\phi'(x)\,dx \nonumber \\
&= \frac{2}{\pi}\gamma_{M}\int_0^{\infty}\|Y(t)\|_F\,dt \nonumber \\
&= \frac{2}{\pi}\gamma_{M}\int_0^{\infty}\|(t^2I+A^2)^{-1}A\|_F\,dt \nonumber \\
&\le \frac{2}{\pi}\gamma_{M}\kappa_2(X)\int_0^{\infty}\|(t^2I+\Lambda^2)^{-1}\Lambda\|_F\,dt \nonumber \\
&= \frac{2}{\pi}\gamma_{M}\kappa_2(X)\int_0^{\infty}\sqrt{\sum_{j=1}^n\frac{|\lambda_j|^2}{|t^2+\lambda_j^2|^2}}\,dt \nonumber \\
&\le \frac{2}{\pi}\gamma_{M}\kappa_2(X)\int_0^{\infty}\sum_{j=1}^n\frac{|\lambda_j|}{|t^2+\lambda_j^2|}\,dt \nonumber \\
&= \frac{2}{\pi}\gamma_{M}\kappa_2(X)\sum_{j=1}^n|\lambda_j|\int_0^{\infty}\frac{1}{|t^2+\lambda_j^2|}\,dt.
\label{eq:Yphibound2}
\end{align}
Now, we explain how to evaluate the integral in the last expression. Let $\lambda_j=\mu_j+{\rm i}\nu_j$, where ${\rm i}=\sqrt{-1}$. Then, the denominator of the integrand can be written as the square root of a real quartic polynomial in $t$. Hence, as is well known, the integral can be transformed into the complete elliptic integral of the first kind:
\begin{equation}
K(k)=\int_0^{\frac{\pi}{2}}\frac{d\theta}{\sqrt{1-k^2\sin^2\theta}}.
\label{eq:elliptic2}
\end{equation}
Details of this transformation are given in Lemma \ref{Lemma2}.
Further, by employing an upper bound on $K(k)$, which is provided in Lemma \ref{Lemma3}, we can obtain an upper bound on this integral. The evaluation proceeds as follows.
\begin{align}
& \int_0^{\infty}\frac{1}{|t^2+\lambda_j^2|}\,dt \nonumber \\
&= \int_0^{\infty}\frac{1}{\sqrt{t^4+2(\mu_j^2-\nu_j^2)t^2+(\mu_j^2+\nu_j^2)^2}}\,dt \nonumber \\
&= \frac{1}{\sqrt{\mu_j^2+\nu_j^2}}K\left(\frac{|\nu_j|}{\sqrt{\mu_j^2+\nu_j^2}}\right) \nonumber \\
&\le \frac{\pi}{2}\cdot\frac{1}{\sqrt{\mu_j^2+\nu_j^2}}\left\{1-\frac{1}{\pi}\log\left(\frac{\mu_j^2}{\mu_j^2+\nu_j^2}\right)\right\} \nonumber \\
&= \frac{\pi}{2}\cdot\frac{1}{|\lambda_j|}\left\{1-\frac{1}{\pi}\log\left(\frac{({\rm Re}(\lambda_j))^2}{|\lambda_j|^2}\right)\right\},
\end{align}
where we used Lemma \ref{Lemma2} in the second equality and Lemma \ref{Lemma3} in the first inequality. Inserting this into Eq.~\eqref{eq:Yphibound2} gives
\begin{equation}
\|E_2\|_F \le \gamma_{M}\kappa_2(X)\left\{n-\frac{2}{\pi}\sum_{j=1}^n\log\left(\frac{|{\rm Re}(\lambda_j)|}{|\lambda_j|}\right)\right\}.
\label{eq:SumError}
\end{equation}

\subsection{Total roundoff error}
The total roundoff error is given as the sum of $E_1$ and $E_2$. From Eqs.~\eqref{TermwiseError} and \eqref{eq:SumError}, we see that the bound on $\|E_1\|_F$ is cubic in $n$ and quartic in $\kappa_2(X)$, while that of $\|E_2\|_F$ is linear in both $n$ and $\kappa_2(X)$. Also, whereas both bounds show singularity when ${\rm Re}(\lambda_j)$ approaches zero, the singularity in the former is $O(1/{\rm Re}(\lambda_j))$, while that of the latter is logarithmic. From these facts, we can say that $E_1$ is dominant in the roundoff error.

\section{Numerical results}
We performed numerical experiments to check the validity of our error bound. In the experiments, we used a PC with the AMD Ryzen 7 3700X Processor (8-Core, 3.59GHz). Our program was written in Python. To compute the matrix products and inverses, we used NumPy. All the computations were performed in double precision arithmetic.

In the experiments, we computed ${\rm sign}(A)$ with the DE-based method, compared the result with that computed by Eq.~\eqref{eq:definition}, and evaluated their difference $E$. In the DE-based method, $h$ was chosen sufficiently small and $N^+$ and $N^-$ were chosen sufficiently large so that the discretization and truncation errors can be neglected. The test matrices were constructed as $A=X\Lambda  X^{\top}$, where $X$ is a nonsingular matrix with a specified condition number $\kappa_2(X)$ and $\Lambda$ is a real random diagonal matrix with a specified condition number $\kappa_2(\Lambda)$. To control the condition number of $X$, we generated $X$ as $X=QDQ^{\top}$, where $Q$ is a random orthogonal matrix and $D$ is a real random diagonal matrix with the condition number $\kappa_2(X)$. The matrix size $n$ was fixed to 100 in the first and second experiments.

In the first experiment, we fixed $\Lambda$ and varied $\kappa_2(X)$ from $10^1$ to $10^6$. The results are shown in Fig.~\ref{fig:fig1}(a). Here, the horizontal axis and the vertical axis denote $\kappa_2(X)$ and $\|E\|_F$, respectively, and log-log plot is used. By regression, the slope is estimated to be 3.886. This is in accordance with our theoretical result, which shows that the dominant component of the roundoff error, $\|E_1\|_F$, is bounded by a quantity proportional to $(\kappa_2(X))^4$.

In the second experiment, we fixed $X$ and varied $\kappa_2(\Lambda)$ from $10^0$ to $10^6$. The result is shown in Fig.~\ref{fig:fig1}(b). In this case, the slope is estimated to be 1.957. This is smaller than expected from our bound \eqref{TermwiseError}, which suggests that $\|E_1\|_F$ is proportional to $(\kappa_F(\Lambda))^3$ when all eigenvalues of $A$ are real. Thus, our bound seems to be somewhat overestimate with respect to $\Lambda$, although it still is a valid upper bound. Note also that the condition number appearing in \eqref{TermwiseError} is $\kappa_F(\Lambda)$ (when all eigenvalues are real), while that specified in the experiment is $\kappa_2(\Lambda)$. Hence, strictly speaking, they are not directly comparable. However, since $\|A\|_2\le \|A\|_F\le \sqrt{n}\|A\|_2$ holds for any matrix $A$, we have
\begin{equation}
\kappa_2(\Lambda) \le \kappa_F(\Lambda) \le n\kappa_2(\Lambda)
\end{equation}
and the comparison is justified if we ignore a polynomial factor in $n$.

\begin{figure}[h]
\centerline{\includegraphics[height=1.4in]{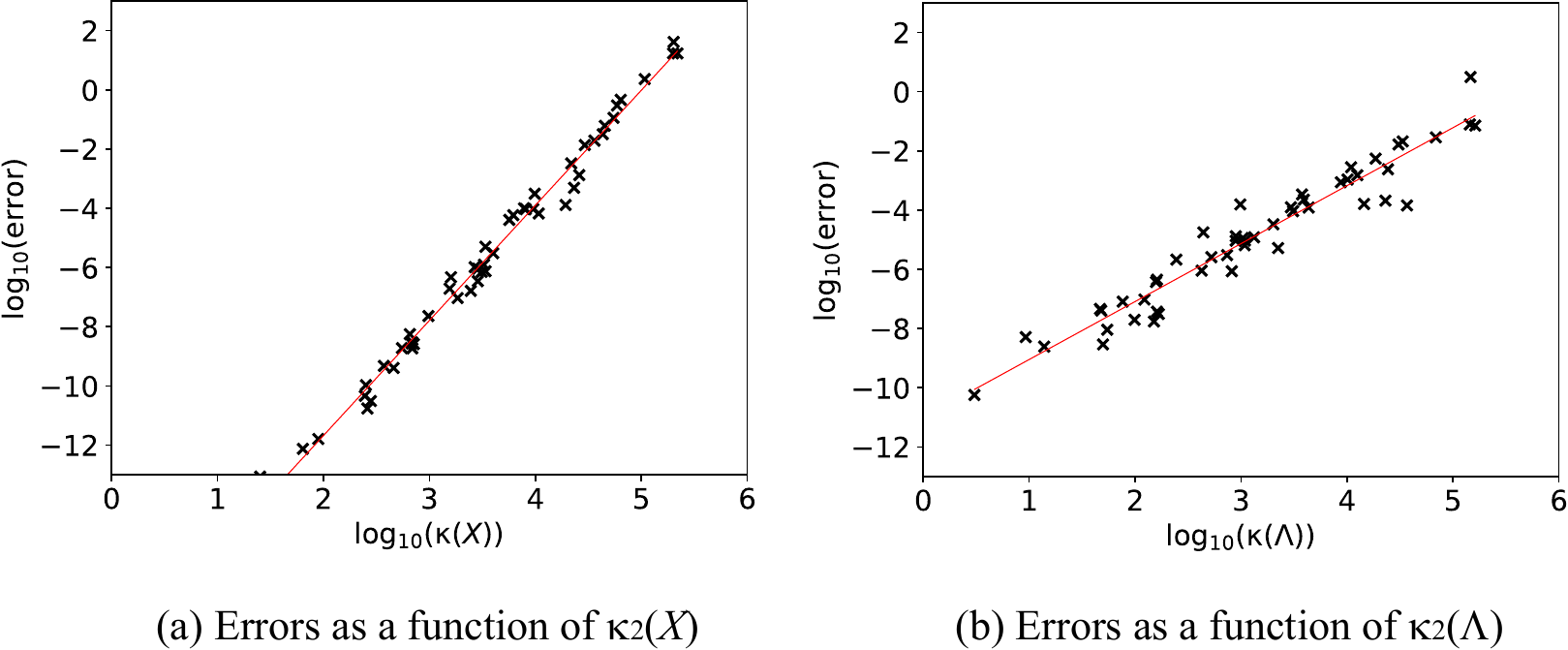}}
\caption{Errors in the computation of ${\rm sign}(A)$.}
\label{fig:fig1}
\end{figure}

In the third experiment, we fixed $\kappa_2(\Lambda)$ and $\kappa_2(X)$ to 100 and varied the matrix size $n$ randomly between 240 and 2560. The result is illustrated in Fig.~\ref{fig:fig2}. Here, the slope is estimated as 0.540. Since the bound \eqref{TermwiseError} on $\|E_1\|_F$ is cubic in $n$ (note that $\gamma_{3n}\simeq 3n{\bf u}$), it is a large overestimate. Essentially, this dependence on $n$ comes from the backward error bound \eqref{eq:DeltaAbound} in the solution of the linear simultaneous equation. However, this is a standard result in the roundoff error analysis of matrix computations and it seems to be difficult to improve it drastically. A possible approach for improvement would be to move from the {\it a priori} error analysis, which we have adopted in this paper, to {\it a posteriori} error analysis, which uses the computed residual to construct the bound. It will be a topic for future research.

\begin{figure}[h]
\centerline{\includegraphics[height=1.35in]{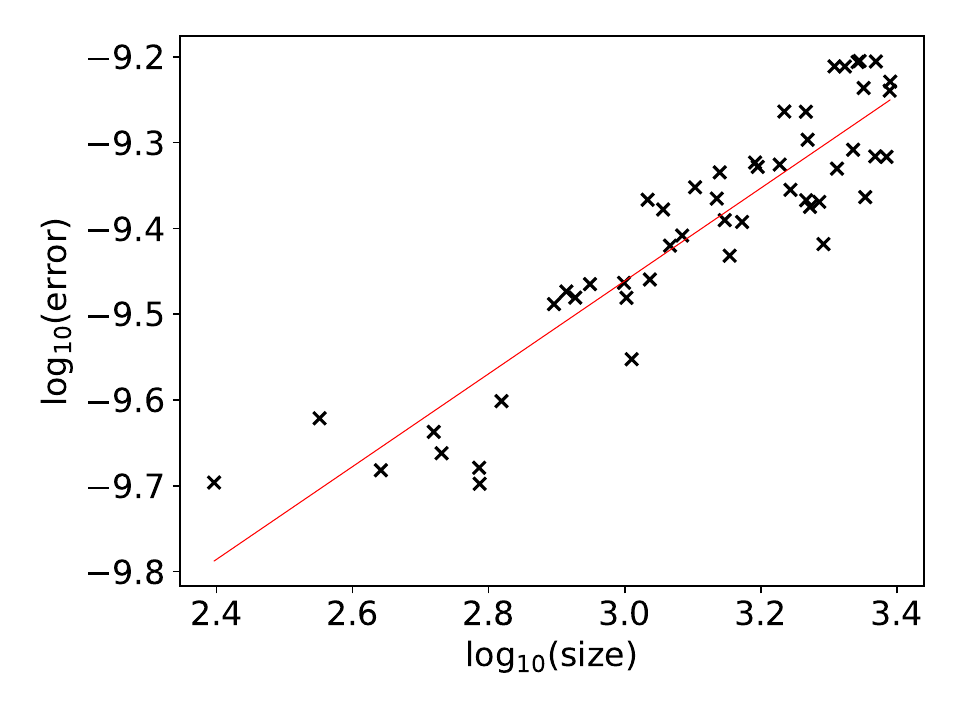}}
\caption{Errors of ${\rm sign}(A)$ as a function of $n$.}
\label{fig:fig2}
\end{figure}

\section{Conclusion}
In this paper, we presented a roundoff error analysis of the DE-based method for the matrix sign function. Our upper bound on the roundoff error is in good accordance with the actual error and partially explains why the accuracy of the method deteriorates when the input matrix is ill-conditioned or highly nonnormal. One possible remedy would be to expand $(t^2 I+A^2)^{-1}$ into partial fractions, by allowing complex arithmetic. This will reduce the dependency of Eq.~\eqref{eq:AF2} on $\kappa_2(X)$ from quadratic to linear, improving the overall dependency of the total error to $(\kappa_2(X))^3$. This will be a topic in our future study.

\section*{Acknowledgment}
This study is supported by JSPS KAKENHI Grant Numbers 19KK0255 and 22KK19772.

\references

\section*{Appendix A}

\begin{Lem}
\label{Lemma1}
Let $\lambda_j\in\mathbb{C}$ and ${\rm Re}(\lambda_j)\ne 0$. Then, Eq.~\eqref{eq:quadraticintegral} holds.
\end{Lem}
\begin{proof}
Let $\lambda_j=\mu_j+{\rm i}\nu_j$, where ${\rm i}=\sqrt{-1}$. Then, the integrand can be rewritten as follows.
\begin{align}
&\frac{1}{|t^2+\lambda_j^2|^2} \nonumber \\
&= \frac{1}{(t^2+\mu_j^2-\nu_j^2)^2+4\mu_j^2\nu_j^2} \nonumber \\
&= \frac{1}{(t^2+\mu_j^2+\nu_j^2)^2-(2\nu_j t)^2} \nonumber \\
&= \frac{1}{(t+\nu_j)^2+\mu_j^2}\cdot\frac{1}{(t-\nu_j)^2+\mu_j^2} \nonumber \\
&= \frac{1}{4\nu_j(\mu_j^2+\nu_j^2)}\left\{\frac{t+2\nu_j}{(t+\nu_j)^2+\mu_j^2}+\frac{-t+2\nu_j}{(t-\nu_j)^2+\mu_j^2}\right\} \nonumber \\
&= \frac{1}{8\nu_j(\mu_j^2+\nu_j^2)}\left\{\frac{2(t+\nu_j)}{(t+\nu_j)^2+\mu_j^2}-\frac{2(t-\nu_j)}{(t-\nu_j)^2+\mu_j^2}\right\} \nonumber \\
& \quad + \frac{1}{4(\mu_j^2+\nu_j^2)}\left\{\frac{1}{(t+\nu_j)^2+\mu_j^2}+\frac{1}{(t-\nu_j)^2+\mu_j^2}\right\}.
\label{eq:lambdaint}
\end{align}
Now, let us choose $L\in\mathbb{R}$ so that $L>\max_{1\le j\le n}|\nu_j|$. Then, if $\nu_j\ge 0$, we have
\begin{align}
\int_{-L}^L\frac{2(t+\nu_j)}{(t+\nu_j)^2+\mu_j^2}\,dt
&= \int_{-L+\nu_j}^{L+\nu_j}\frac{2t'}{{t'}^2+\mu_j^2}\,dt' \nonumber \\
&= \int_{L-\nu_j}^{L+\nu_j}\frac{2t'}{{t'}^2+\mu_j^2}\,dt' \nonumber \\
&= \log\frac{(L+\nu_j)^2+\mu_j^2}{(L-\nu_j)^2+\mu_j^2} \nonumber \\
&\rightarrow 0 \quad (L\rightarrow\infty),
\end{align}
where the second equality is due to the fact that the integrand is an odd function in $t$. The same result holds also when $\nu_j<0$ and also for the integral of $2(t-\nu_j)/\{(t-\nu_j)^2+\mu_j^2\}$. Thus, when the integration interval is $(-\infty, \infty)$, the contribution from the second line from the last of \eqref{eq:lambdaint} is zero.

On the other hand, by letting $t=\mp\nu_j+|\mu_j|\tan\theta$, the contribution from the last line of \eqref{eq:lambdaint} can be computed as follows.
\begin{align}
& \frac{1}{4(\mu_j^2+\nu_j^2)}\left\{\int_{-\infty}^{\infty}\frac{dt}{(t+\nu_j)^2+\mu_j^2} + \int_{-\infty}^{\infty}\frac{dt}{(t-\nu_j)^2+\mu_j^2}\right\} \nonumber \\
&= \frac{1}{4(\mu_j^2+\nu_j^2)}\cdot\frac{2\pi}{|\mu_j|}=\frac{\pi}{2|\lambda_j|^2|{\rm Re}(\lambda_j)|}.
\end{align}
Hence, the lemma is proved.
\end{proof}

\begin{Lem}
\label{Lemma2}
Let $a,c\in\mathbb{R}$ satisfy $a>0$ and $|c|\le a^2$. Then, the following equality holds.
\begin{equation}
\int_0^{\infty}\frac{dx}{\sqrt{x^4+2cx^2+a^4}}=\frac{1}{a}K\left(\frac{\sqrt{a^2-c}}{a\sqrt{2}}\right),
\label{eq:elliptic1}
\end{equation}
where $K(k)$ is the complete elliptic integral of the first kind defined by
\begin{equation}
K(k)=\int_0^{\frac{\pi}{2}}\frac{d\theta}{\sqrt{1-k^2\sin^2\theta}}.
\label{eq:elliptic2b}
\end{equation}
\end{Lem}
\begin{proof}
Following \cite{biquadratic}, we compute the right-hand side of \eqref{eq:elliptic1} and show that it is equal to the left-hand side. Let
\begin{equation}
\cos\theta=\frac{x^2-a^2}{x^2+a^2}.
\end{equation}
Then, the right-hand side is a monotonically increasing function of $x$ in $[0,\infty)$ and $0 \le x <\infty$ corresponds to $\pi \ge \theta > 0$. Since
\begin{equation}
-\sin\theta\,d\theta = \frac{4a^2 x}{(x^2+a^2)^2}\,dx,
\end{equation}
we have
\begin{align}
d\theta&=-\frac{1}{\sin\theta}\cdot\frac{4a^2 x}{(x^2+a^2)^2}\,dx \nonumber \\
&= -\frac{1}{\sqrt{1-\left(\frac{x^2-a^2}{x^2+a^2}\right)^2}}\cdot\frac{4a^2 x}{(x^2+a^2)^2}\,dx \nonumber \\
&= -\frac{2a}{x^2+a^2}\,dx.
\label{eq:dtheta}
\end{align}
Also, note that
\begin{equation}
\sin^2\theta=1-\left(\frac{x^2-a^2}{x^2+a^2}\right)^2=\frac{4a^2 x^2}{(x^2+a^2)^2}.
\label{eq:sinsqtheta}
\end{equation}
Substituting \eqref{eq:dtheta} and \eqref{eq:sinsqtheta} into the right-hand side of \eqref{eq:elliptic1} gives
\begin{align}
& K\left(\frac{\sqrt{a^2-c}}{a\sqrt{2}}\right) \nonumber \\
&= \int_0^{\frac{\pi}{2}}\frac{d\theta}{\sqrt{1-\frac{a^2-c}{2a^2}\sin^2\theta}} \nonumber \\
&= \frac{1}{2}\int_0^{\pi}\frac{d\theta}{\sqrt{1-\frac{a^2-c}{2a^2}\sin^2\theta}} \nonumber \\
&= \frac{1}{2}\int_0^{\infty}\frac{1}{\sqrt{1-\frac{a^2-c}{2a^2}\cdot\frac{4a^2 x^2}{(x^2+a^2)^2}}}\cdot\frac{2a}{x^2+a^2}\,dx \nonumber \\
&= a\int_0^{\infty}\frac{dx}{\sqrt{x^4+2cx^2+a^4}}.
\end{align}
Hence, the lemma is proved.
\end{proof}

\begin{Lem}
\label{Lemma3}
For $k\in(-1,1)$, the following inequality holds.
\begin{equation}
K(k) \le \frac{\pi}{2}\left\{1-\frac{1}{\pi}\log(1-k^2)\right\}.
\end{equation}
\end{Lem}
\begin{proof}
This inequality can be shown as a special case of the results given in \cite{Andras10}. For real numbers $a,b,c$ and $k$ satisfying $c\ne 0, -1, -2, \ldots$ and $k\in(-1,1)$, the Gauss hypergeometric function is defined by
\begin{equation}
_2 F_1(a,b,c,k) = \sum_{n\ge 0}\frac{(a)_n(b)_n}{(c)_n}\cdot\frac{k^n}{n!},
\end{equation}
where $(a)_n=\Gamma(a+n)/\Gamma(a)=a(a+1)\cdots(a+n-1)$. It is known that $K(k)$ can be expressed with $_2 F_1(a,b,c,k)$ as follows.
\begin{equation}
K(k)=\frac{\pi}{2}\cdot\,_2F_1\left(\frac{1}{2},\frac{1}{2},1,k^2\right).
\end{equation}
Now, let us consider a generalized complete elliptic integral of the first kind defined by
\begin{equation}
K_a(k)=\frac{\pi}{2}\cdot\,_2F_1\left(a,1-a,1,k^2\right).
\end{equation}
Andra\'{s} and Baricz \cite[Eq. (7)]{Andras10} showed that $K_a(k)$ has the following lower and upper bounds when $a,k\in(0,1)$.
\begin{align}
& \frac{\pi}{2}\left\{1-\frac{\sin(\pi a)}{\pi}\left(\frac{1}{k^2}\log(1-k^2)+1\right)\right\} \nonumber \\
&< K_a(k) < \frac{\pi}{2}\left\{1-\frac{\sin(\pi a)}{\pi}\log(1-k^2)\right\}.
\end{align}
By setting $a=1/2$, we immediately obtain the following bound on $K(k)$ that is valid when $k\in(0,1)$.
\begin{equation}
K(k) < \frac{\pi}{2}\left\{1-\frac{1}{\pi}\log(1-k^2)\right\}.
\end{equation}
Noting that $K(0)=\frac{\pi}{2}$ and replacing $<$ with $\le$, we obtain an inequality that is valid also for $k=0$. Hence, the lemma is proved.
\end{proof}

\end{document}